\documentclass[12pt]{amsart}

\usepackage{amsmath,amssymb,amsbsy,amsfonts,amsthm,latexsym,
                     amsopn,amstext,amsxtra,euscript,amscd,mathrsfs}
\usepackage[colorlinks,linkcolor=blue,anchorcolor=blue,citecolor=blue]{hyperref}

\newtheorem{thm}{Theorem}

\newtheorem{lemma}[thm]{Lemma}

\newtheorem{theorem}[thm]{Theorem}
\newtheorem{cor}[thm]{Corollary}

\newtheorem{rem}[thm]{Remark}

\numberwithin{equation}{section}
\numberwithin{thm}{section}

\def\cD{{\mathcal D}}

\def\cF{{\mathcal F}}
\def\cG{{\mathcal G}}
\def\cH{{\mathcal H}}
\def\cI{{\mathcal I}}
\def\cJ{{\mathcal J}}

\def\cL{{\mathcal L}}

\def\cP{{\mathcal P}}

\def\cT{{\mathcal T}}
\def\cU{{\mathcal U}}
\def\cV{{\mathcal V}}

\def\cX{{\mathcal X}}

\def\A{{\mathbb A}}

\def\F{{\mathbb F}}

\def\K{{\mathbb K}}

\def\P{{\mathbb P}}
\def\Q{{\mathbb Q}}

\def\Z{{\mathbb Z}}

\def\ep{{\mathbf{e}}_p}
\def\e{{\mathbf{e}}}
\def\eV{{\mathbf{e}}_V}

\def\ssum{\mathop{\sum\, \sum}}

\def\\{\cr}
\def\({\left(}
\def\){\right)}
\def\[{\left[}
\def\]{\right]}
\def\<{\langle}
\def\>{\rangle}
\def\fl#1{\left\lfloor#1\right\rfloor}
\def\rf#1{\left\lceil#1\right\rceil}

\def\mand{\qquad\mbox{and}\qquad}

\def\GL{\operatorname{GL}}

\def\sym{\operatorname{sym}}
\def\Sym{\operatorname{Sym}}
\def\ind{\operatorname{ind}}
\def\ord{\operatorname{ord}}

\begin{document}

\title[The Sato--Tate Conjecture In Thin Families of Curves]{The Sato--Tate Distribution in Thin Parametric Families of Elliptic Curves}

\author[R. de la Bret\`eche]{R\'egis de la Bret\`eche}
\address{Institut de Math\'ematiques de Jussieu-PRG, Universit\'e Paris Diderot,
Sorbonne Paris Cit\'e, UMR 7586, Case 7012, F-75013 Paris, France}
\email{regis.de-la-breteche@imj-prg.fr}

\author[M. Sha]{Min Sha}
\address{School of Mathematics and Statistics, University of New South Wales,
 Sydney, NSW 2052, Australia}
\email{shamin2010@gmail.com}

\author[I. E. Shparlinski]{Igor E. Shparlinski}
\address{School of Mathematics and Statistics, University of New South Wales,
 Sydney, NSW 2052, Australia}
\email{igor.shparlinski@unsw.edu.au}

\author[J. F. Voloch]{Jos{\'e} Felipe Voloch}
\address{Department of Mathematics, University of Texas, Austin, TX~78712, USA}
\email{voloch@math.utexas.edu}

\date{}

\begin{abstract}
We obtain new results concerning the  Sato--Tate conjecture on
the distribution of Frobenius traces
over single and double parametric families of elliptic curves.
We consider these curves  for values of parameters
having prescribed arithmetic structure: product sets, geometric
progressions, and  most significantly prime numbers.  In particular,
some families are much thinner than the ones previously studied.
\end{abstract}

\maketitle

\section{Introduction}
\label{sec:intro}

\subsection{Background and motivation}

For polynomials $f(Z), g(Z) \in \Z[Z]$ satisfying
\begin{equation}
\label{eq:Nondeg}
\Delta(Z) \neq 0 \mand j(Z)  \not\in \Q,
\end{equation}
where
$$
\Delta(Z)   = -16 (4f(Z)^3 + 27g(Z)^2)\quad\text{and}\quad
j(Z)    = \frac{-1728(4f(Z))^{3}}{\Delta(Z)}
$$
are the {\it discriminant\/} and {\it $j$-invariant\/}, respectively,
we consider
the   elliptic curve
\begin{equation}
\label{eq:Family AB}
E(Z) : \quad Y^2 = X^3 + f(Z)X + g(Z)
\end{equation}
over the function field $\Q(Z)$; see~\cite{Silv} for a general background on elliptic curves.
In particular,  we refer
to~\cite{Silv}  for the notions of the {\it conductor\/} $N_E$ of an elliptic curve $E$
and {\it CM-curves\/}.

There exists an extensive literature on investigating the properties of
the specialisations $E(t)$ modulo consecutive primes $p\le x$
for a growing parameter $x$ and for  the parameter $t$ that
runs through some interesting sets $\cT$ of integer or rational
numbers,  see~\cite{ShaShp} for a survey and some recent results;
a short outline is also given in Section~\ref{sec:prev res}.

More precisely, given an elliptic curve $E$ over $\Q$ we   denote by $E_p$
the reduction of $E$ modulo $p$. In particular, we use
$E_p(\F_p)$ to denote the group of $\F_p$-rational
points on $E_p$, where $\F_p$ is the finite field of
$p$ elements.  We also define, as usual,
the {\it Frobenius trace\/} $a_p(E)=p+1-\#E_p(\F_p)$.

There are several possible scenarios in the study of the curves from the
family~\eqref{eq:Family AB} (or similar family~\eqref{eq:Family uv}) and their reductions:
\begin{itemize}
\item One can fix a curve and vary   the prime $p$. This is usually called the
{\it horizontal\/} aspect (and is typically very hard to study).

\item One can fix  a prime $p$ and consider the curves $E_p(t)$ for all values of the parameter $t$
from some ``interesting'' set $\cT$.  This is usually called the
{\it vertical\/} aspect.

\item One can vary both the prime $p$ and  the curves $E_p(t)$ for $t\in \cT$,
 we call this the {\it mixed\/} aspect.

\end{itemize}

Clearly the mixed aspect combines both horizontal and vertical aspects
and often leads to results which are not possible within either of them.

Recall that by the Hasse bound (see~\cite{Silv}), we can define the {\it Frobenius angle\/} $\psi_p(E) \in [0, \pi]$ via the
identity
\begin{equation}
\label{eq:ST angle}
\cos \psi_p(E) = \frac{a_p(E)}{2\sqrt{p}}.
\end{equation}
Then, in general terms, the \textit{Sato--Tate  conjecture} predicts that the distribution of the angles $\psi_p(E)$ is governed by the
 {\it Sato--Tate  density\/}
\begin{equation}
\label{eq:ST dens}
\mu_{\tt ST}(\alpha,\beta) = \frac{2}{\pi}\int_\alpha^\beta
\sin^2\vartheta\, {\rm d} \vartheta = \frac{2}{\pi}\int_{\cos \beta}^{\cos \alpha}
(1-z^2)^{1/2}\, {\rm d} z,
\end{equation}
where $[\alpha,\beta] \subseteq [0,\pi]$.

In the   vertical aspect when $p$ is fixed and $E$ is chosen at random
from the set of all elliptic curves over $\F_p$, this has been shown by Birch~\cite{Birch}.

 The  horizontal aspects is much harder and the Sato--Tate conjecture
has  been settled only quite recently  in the series of works of Barnet-Lamb,  Geraghty,
Harris and Taylor~\cite{B-LGHT},
Clozel,  Harris and Taylor~\cite{CHT},
Harris,  Shepherd-Barron and  Taylor~\cite{HS-BT}, and Taylor~\cite{Taylor2008}.
In particular, given a non-CM elliptic curve $E$ of conductor $N_E$,  for the number $\pi_{E}(\alpha,\beta;x)$ of
primes $p \le x$ (with $p \nmid N_E$) for which
$\psi_p(E) \in[\alpha, \beta] \subseteq [0,\pi]$, we have
$$
\pi_{E}(\alpha,\beta;x) \sim
\mu_{\tt ST}(\alpha,\beta) \cdot\frac{x}{\log x}
$$
as $x \to \infty$. However, this asymptotic formula lacks an explicit error term.

Here, we are mostly interested in the vertical and mixed aspects however extended only to
some families of curves, such as~\eqref{eq:Family AB} specialised for
parameters $t$ from a sparse set $\cT$, rather than for all curves over $\F_p$
as in~\cite{Birch}.
We  show for several such families of curves and sets $\cT$ that  the Frobenius angles
are also distributed according to the Sato--Tate  density. Several results of this type are
already
known, however mostly for sets $\cT$ of integers having some {\it additive structure\/} such as intervals
of consecutive integers or sumsets, see~\cite{ShaShp,Shp3}. Here, thanks to
Lemma~\ref{lem:Mich bound} we consider a new
class of sets $\cT$ which are defined by some {\it multiplicative conditions\/} such as
primes or multiplicative subgroups of $\F_p^*$. In particular, such sets tend to be
much sparser than the sets studied in previous works.

As an application of these results, we also consider the mixed situation when both the curve
and the prime vary.

Similar questions have also been studied for some other families such as~\eqref{eq:Family uv}
below,
with $u$ and $v$ in some subsets of $\F_p$, see Section~\ref{sec:prev res} for
more details.

\subsection{Previous  results}
\label{sec:prev res}

The idea of studying the properties of reduction $E_p$ for $p\le x$
on average  over a family of elliptic curves $E$ is due to Fouvry and  Murty~\cite{FoMu},
who have considered the frequency of vanishing $a_p(E_{u,v}) = 0$
for  the family of curves
\begin{equation}
\label{eq:Family uv}
E_{u,v}:\ Y^2 = X^3 + uX + v,
\end{equation}
where the integers $u$ and $v$ satisfy the inequalities
$|u| \le U$, $|v|\le V$. The result  of~\cite{FoMu} has been extended
to other values of $a_p(E_{u,v})$ by
David and Pappalardi~\cite{DavPapp} and
Baier~\cite{Baier1}, see also~\cite{Baier2}. This corresponds to
the \textit{Lang--Trotter conjecture}, see~\cite{Lang}, on average
over a family of curves~\eqref{eq:Family uv}.

The above results and methods of~\cite{Baier1,Baier2,DavPapp,FoMu}
can also be used to establish the Sato--Tate conjecture on average
for the family~\eqref{eq:Family uv}, see also~\cite{BaZha,Shp3}.
However, Banks and Shparlinski~\cite{BaSh} have shown that using
a different approach, based on
bounds of multiplicative character sums and the large sieve
inequality (instead of the exponential sum technique employed in~\cite{FoMu}), one can establish the Sato--Tate conjecture
on average for the curves~\eqref{eq:Family uv}; see also \cite{Shp2}. 

Furthermore, Shparlinski~\cite{Shp3} has established the Sato--Tate conjecture on average
for more general families of the form $E_{f(u),g(v)}$
with  integers $|u| \le U$, $|v|\le V$. 
Recently,  Sha and Shparlinski~\cite{ShaShp} have established the Sato--Tate conjecture on average for the families of curves $E(u+v)$, where $u,v$ both run through some subsets of $\{1,2,\ldots,T\}$, or both run over the set $\cF(T)$:
$$
\cF(T) = \{u/v \in \Q~:~\gcd(u, v) = 1, \,1 \le
u,v \le T\}.
$$
Most recently, using Lemma~\ref{lem:Mich bound} Sha and Shparlinski~\cite{ShaShp2} 
have also established the Sato--Tate conjecture on average for the families of curves $E(uv)$, where $u,v$ both run over the set $\cF(T)$. 
Finally, Cojocaru and Hall~\cite{CojHal} have studied the family
of curves~\eqref{eq:Family AB} and obtained an upper bound on the
frequency of the event $a_p(E(t)) = a$ for a fixed integer $a$,
when the parameter $t$ runs through the set $\cF(T)$.
Cojocaru and Shparlinski~\cite{CojShp}
have improved~\cite[Theorem~1.4]{CojHal}, which then has been
further improved by Sha  and Shparlinski~\cite{ShaShp}.

\subsection{Distribution of Frobenius traces and ranks}
\label{sec:Frob Rank}

Our motivation also comes from  the so-called  \textit{explicit formulas}, which can be found
in the works of Mestre~\cite{Mest1,Mest2}
that link the behaviour of Frobenius traces on consecutive
primes (that is, the horizontal aspect)
and the rank of the corresponding elliptic curve. This link has been used by
Fouvry and Pomykala~\cite{FoPo},  Michel~\cite{Mich} and Silverman~\cite{Silv1}
to estimate
the average rank in some families of elliptic curves.
For example, Michel~\cite[Theorem~1.3]{Mich} and, in a stronger
form, Silverman~\cite[Theorem~0.1]{Silv1}
 give  explicit bounds on the average rank of the
curves $E(t)$ in the family~\eqref{eq:Family AB} with $t$ running through all integers
of the interval $[-T,T]$ with $\Delta(t) \ne 0$ as $T \to \infty$. This direction is
naturally related to the mixed aspect.
The results here also can be compared to the recent result of
Bhargava and Shankar \cite{BS1}~that the average rank of all elliptic curves over $\Q$ (when ordered by height) is bounded,   see also~\cite{Kow,Poon} for outlines  of several other related results.

Combining our estimates with the approaches of~\cite{FoPo,Mich,Silv1},
one may obtain upper bounds on average ranks of families of curves
with parameters  from sets of prescribed multiplicative structure, such as primes, geometric progressions, and product sets.

\subsection{General notation}
\label{sec:notation}
Here we use the Landau symbol $O$ and the Vinogradov symbol $\ll$. We recall that the assertions $A=O(B)$ and $A\ll B$ are both equivalent to the inequality $|A|\le cB$ with some absolute constant $c>0$.

Throughout the paper  the implied constants  may, where obvious, depend on the polynomials $f$ and $g$ in~\eqref{eq:Family AB} and the real positive parameter $\varepsilon$, and are absolute otherwise.
Occasionally they also depend on the integer parameter $\lambda$ which we indicate
as  $O_\lambda$ and $\ll_\lambda$.

As usual, $A=o(B)$ means that $A/B\to 0$ and $A \sim B$  means that $A/B\to 1$.

Furthermore,  the letters $\ell$ and $p$ always denote a prime number, and as usual, we use $\pi(x)$ to denote the number of primes $p\le x$.

We always assume that the elements of $\F_p$
are represented by the set $\{0, \ldots, p-1\}$ and thus
we switch freely between the equations in $\F_p$ and congruences
modulo $p$ (for example, compare the definitions of
$N_p(\alpha,\beta; \cG)$, $N_p(\alpha,\beta;  \cU,\cV)$  and $Q_p(\alpha,\beta; L)$
in Sections~\ref{sec: vert results} and~\ref{sec:mix results} below).

\section{Main Results}

\subsection{Our approach}
\label{sec:approach}

In this paper, we consider the  Sato--Tate conjecture on average for the polynomial family~\eqref{eq:Family AB} of elliptic curves when the variable $Z$
is specialised to a parameter $t$ from sets of prescribed
multiplicative structure, such as subgroups of $\F_p^*$, prime numbers, and geometric progressions.

We believe that these are the first known  results that involve
such sets of parameters.

To derive our results we introduce several new ideas, such as a version of
a result of Michel~\cite[Proposition~1.1]{Mich} with multiplicative characters
(see Lemma~\ref{lem:Mich bound}).
This is combined with a technique
of Niederreiter~\cite[Lemma~3]{Nied}.
To study the curves~\eqref{eq:Family AB} for specialisations
at consecutive primes, we also estimate some bilinear sums
(which maybe of independent interest) and combine this bound with  the Vaughan
identity~\cite{Vaughan,Vau}.

We are now able to give exact formulations
of our results. We always assume that the polynomials
$f$ and $g$ in~\eqref{eq:Family AB}
are fixed and so we do not include them in our notation.
We also often impose the following modulo $p$
analogue of the condition~\eqref{eq:Nondeg}:
\begin{equation}
\label{eq:Nondeg p}
\Delta(Z) \not \equiv  0 \pmod p \quad \text{and}\quad j(Z)~\text{is not constant modulo $p$}.
\end{equation}

\subsection{Our  results in the vertical aspect}
\label{sec: vert results}
Here, we fix an interval $[\alpha,\beta] \subseteq [0,\pi]$, we also fix an arbitrary prime $p$ for 
Theorems~\ref{thm:Subgroup}, \ref{thm:ProductSet}, \ref{thm:Prime1} and~\ref{thm:Prime2}.

Given a multiplicative subgroup $\cG\subseteq \F_p^*$, we
denote by $N_p(\alpha,\beta; \cG)$ the number of $w\in \cG$
for which $\Delta(w) \ne 0$ and  $\psi_p(E(w)) \in [\alpha,\beta]$.

\begin{thm}
\label{thm:Subgroup}
Suppose that the polynomials $f(Z), g(Z) \in \Z[Z]$
satisfy~\eqref{eq:Nondeg p}.
Then for any  subgroup
$\cG\subseteq \F_p^*$ of order $r$,  uniformly over
$[\alpha,\beta] \subseteq [0,\pi]$, we have
$$
N_p(\alpha,\beta; \cG)=\mu_{\tt ST}(\alpha,\beta)r   +O\(r^{1/2} p^{1/4}\).
$$
\end{thm}

We remark that noticing the trivial bound $N_p(\alpha,\beta; \cG)\le r$,  the result in Theorem~\ref{thm:Subgroup} is non-trivial when $r>r^{1/2} p^{1/4}$,
namely $r>\sqrt{p}$.

Similarly, given  two sets
$\cU, \cV  \subseteq \F_p^*$, we
denote by $N_p(\alpha,\beta;  \cU,\cV)$ the number of
$(u,v) \in \cU \times\cV$
for which $\Delta(uv) \ne 0$ and  $\psi_p(E(uv)) \in [\alpha,\beta]$.

\begin{thm}
\label{thm:ProductSet}
Suppose that the polynomials $f(Z), g(Z) \in \Z[Z]$
satisfy~\eqref{eq:Nondeg p}.
Then for any non-empty subsets
$\cU, \cV  \subseteq \F_p^*$,   uniformly over
$[\alpha,\beta] \subseteq [0,\pi]$, we have
$$
N_p(\alpha,\beta;  \cU,\cV)=\mu_{\tt ST}(\alpha,\beta) \# \cU  \# \cV  +O\( \(\# \cU  \# \cV\)^{3/4} p^{1/4}\).
$$
\end{thm}

Note that the result in Theorem~\ref{thm:ProductSet} is non-trivial when $\# \cU  \# \cV > p$.

Furthermore, given an integer $L$,  we
denote by $Q_p(\alpha,\beta; L)$ the number of primes $\ell \le L$
for which $\Delta(\ell) \not \equiv 0 \pmod p$ and
$\psi_p(E(\ell)) \in [\alpha,\beta]$.

First we record the following result whose proof rests on the recent work of Fouvry, Kowalski and
Michel~\cite{FKM}. 

\begin{thm}
\label{thm:Prime1}
Suppose that the 
 polynomials $f(Z), g(Z) \in \Z[Z]$ 
satisfy~\eqref{eq:Nondeg p}. Then, for any $\varepsilon > 0$,  there exists some  $\rho>0$  such that  
 for any  integer $L \ge p^{3/4 + \varepsilon}$, uniformly over  $[\alpha,\beta] \subseteq [0,\pi]$ 
we have
$$
 Q_p(\alpha,\beta; L)  =\( \mu_{\tt ST}(\alpha,\beta)  + O\(p^{-\rho}\)\) \pi(L) .
$$
\end{thm}

We remark that it  is apparent from the proof of Theorem~\ref{thm:Prime1}, 
given in Section~\ref{sec:proof prime}, that we always have $\rho < 1/48$ within 
this approach (and perhaps much smaller depending on the constant $A$ in 
Lemma~\ref{lem:Prime1}). Hence, for large $L$, using different arguments we derive a stronger bound with an explicit 
saving. 

\begin{thm}
\label{thm:Prime2}
Suppose that the polynomials $f(Z), g(Z) \in \Z[Z]$
satisfy~\eqref{eq:Nondeg p}.
Then for any  integer $L \ge 3$, uniformly over  $[\alpha,\beta] \subseteq [0,\pi]$,
we have
\begin{equation*}
\begin{split}
 Q_p(\alpha,\beta; L)  = & \mu_{\tt ST}(\alpha,\beta) \pi(L) \\
&  + O\(
\( L p^{-1/4}+ L^{11/12}+ L^{3/4}p^{1/4}\) L^{c/\log\log L}\),
\end{split}
\end{equation*}
for some absolute constant $c>0$.
\end{thm}

One should note that the result in Theorem~\ref{thm:Prime2} is
non-trivial only  when $L > p$. We also remark that Theorem \ref{thm:Prime2} is better than Theorem \ref{thm:Prime1} when $L$ is large but not too large compared with $p$ (for example, polynomially lower and upper bounded in terms of $p$).  
Otherwise, Theorem \ref{thm:Prime1} may be better, such as $L \ge \exp(p)$.

\subsection{Our results in the mixed aspect}
\label{sec:mix results}

Here, we establish the Sato--Tate conjecture on average for some families of elliptic curves
which have never been studied in the literature.

Recall that for any integer $t$ with $\Delta(t)\ne 0$,
we use $\pi_{E(t)}(\alpha,\beta; x)$
to denote the number of primes $p \le x$ with $p \nmid N_{E(t)}$ (or equivalently, $\Delta(t) \not\equiv 0 \pmod p$, see Section~\ref{sec:not}) and
$\psi_p(E(t))\in[\alpha,\beta]$.
First, we get an analogue of~\cite[Theorems 10]{ShaShp}.

\begin{theorem}
\label{thm:S-T Setprod}
Suppose that the polynomials $f(Z), g(Z) \in \Z[Z]$
satisfy~\eqref{eq:Nondeg}, and non-empty sets  of integer $\cU, \cV \subseteq [1,x]$, $x\ge 2$, are arbitrary.
Then, uniformly over  $[\alpha,\beta] \subseteq [0,\pi]$,  we have
$$
\frac{1}{ \pi(x) \#\cU\#\cV}\sum_{\substack{u \in \cU, v\in \cV \\ \Delta(uv) \ne 0}}
\pi_{E(uv)}(\alpha,\beta; x)=\mu_{\tt ST}(\alpha,\beta) +O\(\(\frac{x}{\# \cU \# \cV}\)^{1/4}\).
$$
\end{theorem}

In Theorem~\ref{thm:S-T Setprod}, if $x=o(\# \cU \# \cV)$, then we indeed establish the Sato-Tate conjecture on average for the corresponding family of elliptic curves.

Now, we establish the Sato--Tate conjecture on average when the parameter runs through some new kinds of subsets which have not been studied before.

First, we establish the Sato--Tate conjecture on average with a parameter $t$ from a geometric progression. Namely, given
 integers $t, \lambda$ with $|\lambda|\ge 2$ and real $x \ge 2$,
 we use $\pi_{\lambda}(\alpha,\beta; t, x)$
to denote the number of $p \le x$ with $\Delta(\lambda^t) \not\equiv 0 \pmod p$ and
$\psi_p\(E\(\lambda^t\)\)\in[\alpha,\beta]$. We also define the {\it Erd\H{o}s constant\/}
\begin{equation}
\label{eq:Erdos}
\delta = 1- \frac{1+\log \log 2}{\log 2} =0.086071\ldots.
\end{equation}

Then we have:

\begin{thm}
\label{thm:S-T GeomProgr}
Suppose that the polynomials $f(Z), g(Z) \in \Z[Z]$
satisfy~\eqref{eq:Nondeg}.
Then for any real $x \ge 3$ and integer $$T \ge x^{1/2}
(\log x)^{1+3\delta/2}(\log \log x)^{ 9/4},$$ uniformly over  $[\alpha,\beta] \subseteq [0,\pi]$, we have
\begin{equation*}
\begin{split}
 \frac{1}{\pi(x) T}\sum_{\substack{1\le t \le T \\ \Delta(\lambda^t) \ne 0}} & \pi_{\lambda}(\alpha,\beta; t, x)\\
& =\mu_{\tt ST}(\alpha,\beta)   +O_{\lambda}\((\log x)^{-3\delta/4}(\log \log x)^{-9/8}\),
\end{split}
\end{equation*}
where $\delta$ is given by~\eqref{eq:Erdos}.
\end{thm}

 It is possible to get a better error term if one averages on $\lambda$,  however we do not address this question here.

 For families parametrised by primes, 
  averaging the estimate in Theorem~\ref{thm:Prime1}, we get the following result.

  \begin{thm}
\label{thm:S-T Primes 1}
Suppose that the polynomials $f(Z), g(Z) \in \Z[Z]$
satisfy~\eqref{eq:Nondeg}. Then, for any $\varepsilon>0$ with the constant $\rho>0$ 
defined in Theorem~\ref{thm:Prime1}, for any real $x \ge 2$ and integer $L$ with  $L \ge x^{3/4 + \varepsilon}$, 
 uniformly over  $[\alpha,\beta] \subseteq [0,\pi]$ we have
\begin{equation*}
\begin{split}
 \frac{1}{\pi(x) \pi(L)} \sum_{\substack{\textrm{prime } \ell \le L\\
\Delta(\ell) \ne 0}} &
\pi_{E(\ell)}(\alpha,\beta; x)  =\mu_{\tt ST}(\alpha,\beta)    +
 O\(x^{-\rho}\).
\end{split}
\end{equation*}
\end{thm}

Finally, for large values of $L$, averaging the estimate in Theorem~\ref{thm:Prime2}, we
derive  the following explicit result.

\begin{thm}
\label{thm:S-T Primes 2}
Suppose that the polynomials $f(Z), g(Z) \in \Z[Z]$
satisfy~\eqref{eq:Nondeg}.
Then, for any real $x \ge 2$ and any integer $L\ge 3$, 
 uniformly over  $[\alpha,\beta] \subseteq [0,\pi]$ we have
\begin{equation*}
\begin{split}
& \frac{1}{\pi(x) \pi(L)} \sum_{\substack{\textrm{prime } \ell \le L\\
\Delta(\ell) \ne 0}} 
\pi_{E(\ell)}(\alpha,\beta; x)  \\
&\qquad\qquad =\mu_{\tt ST}(\alpha,\beta)    +
O\(\(x^{-1/4} + L^{-1/12}+ L^{-1/4}x^{1/4}\) L^{c/\log\log L}\), 
\end{split}
\end{equation*}
for some absolute constant $c>0$.
\end{thm}

Note that in Theorem \ref{thm:S-T Primes 2}, if $L\ge x^{1 + \varepsilon}$ and $L$ can be polynomially upper bounded in terms of $x$, then the error term tends to zero when $x$ goes to infinity.

\section{Preliminaries}
\label{Preliminary}

\subsection{Primes of good reduction}
\label{sec:not}

We start with the observation that the condition~\eqref{eq:Nondeg}
(over any field $\K$ of characteristic $p >3$)
implies that $\Delta(Z)\in \K[Z]$ is not a constant polynomial. Indeed, if
$\Delta(Z) = c\ne 0$ for some $c \in \K$, then $f(Z)$ and $g(Z)$ have
no common roots.  Since $j(Z)$ is not constant, both $f$ and $g$ are also not constant. Now,  considering the derivative $\Delta'(Z)  = 0$, we easily see that $f$ and $g$ must have common roots, which leads to a contradiction.

Thus, apart from at most finitely many primes $p$, the condition~\eqref{eq:Nondeg} 
in fact includes the condition~\eqref{eq:Nondeg p}, 
although usually the latter one is stronger. 

For  $t \in \Q$, let $N(t)$ denote the conductor of the  specialisation of
$E(Z)$ at $Z = t$. We always consider rational numbers in
the form of irreducible fraction.

Note that for $t\in \Q$, the discriminant $\Delta(t)$ may be a rational number. However, we know that the elliptic curve $E(t)$ has good reduction at prime $p$ if and only if $p$ does not divide both the numerator and denominator of $\Delta(t)$; see~\cite[Chapter VII, Proposition 5.1 (a)]{Silv}.
So, we can say that for any prime $p$, $p \nmid N(t)$ (that is, $E(t)$ has good reduction at $p$) if and only if $\Delta(t)\not\equiv 0 \pmod p$ (certainly, it first requires that $p$ does not divide the denominator of $\Delta(t)$).

\subsection{Preparations for distribution of angles}
\label{sec:preparation}

For $m$ arbitrary elements $w_1,\ldots,w_m \in [-1,1]$ (not necessarily distinct) and an arbitrary subinterval $\cJ \subseteq [-1,1]$, let $A(\cJ;m)$ be the number of integers $i$, $1\le i \le m$, with $w_i \in \cJ$.  For any $-1\le a < b \le 1$, define the  function
$$
G(a,b)=\frac{2}{\pi} \int_{a}^{b} (1-z^2)^{1/2} \, {\rm d} z.
$$
We also recall the Chebyshev polynomials $U_n$ of the second kind, on $[-1,1]$ they are defined by
$$
U_n(z) = \frac{\sin((n+1)\arccos z)}{(1-z^2)^{1/2}} \quad \textrm{for $z\in [-1,1]$},
$$
where $n$ is a nonnegative integer. In particular, for $\vartheta \in [0,\pi]$, we have
$$
U_n(\cos \vartheta) = \sym_n (\vartheta),
$$
where
\begin{equation}
\label{eq:sym_n}
\sym_n (\vartheta) =\frac{\sin \((n+1)\vartheta\) }{\sin \vartheta}.
\end{equation}

The following result is exactly from \cite[Lemma 17]{ShaShp}, which is a direct consequence of a result of Niederreiter~\cite[Lemma~3]{Nied}.
\begin{lemma}
\label{lem:Nied}
For any integer $k\ge 1$, we have
$$
\max_{-1\le a < b \le 1} \left| A([a,b];m) - mG(a,b) \right|
\ll \frac{m}{k} + \sum_{n=1}^{k}\frac{1}{n} \left| \sum_{i=1}^{m} U_n(w_i) \right|.
$$
\end{lemma}

\begin{cor}
\label{cor:ST Discrep}
Given $m$ arbitrary angles $\psi_1,\ldots,\psi_m \in [0, \pi]$ (not necessarily distinct), assume that
for some  constant $A>0$ 
 we have
$$
 \left| \sum_{i=1}^{m} \sym_n(\psi_i) \right| \le n^A \sigma
$$
for every integer $n \ge 1$.
Then, uniformly over  $[\alpha,\beta] \subseteq [0,\pi]$, we have
$$
\# \{1\le i \le m : \psi_i\in [\alpha,\beta]\} =  \mu_{\tt ST}(\alpha,\beta) m
+ O\(m^{A/(A+1)} \sigma^{1/(A+1)}\).
$$
\end{cor}

\begin{proof}
We apply Lemma~\ref{lem:Nied}
to the sequence $\cos \psi_1, \ldots, \cos \psi_m$ and obtain
$$
\max_{[\alpha,\beta] \subseteq [0,\pi]}\Big|\# \{1\le i \le m~:~\psi_i\in [\alpha,\beta]\}
-   \mu_{\tt ST}(\alpha,\beta)m \Big|
\ll  \frac{m}{k} +  k^A \sigma. 
$$
Now, assume that $\sigma < m$
as otherwise the result is trivial. 
Then, we conclude the proof by taking 
$k = \rf{(m/\sigma)^{1/(A+1)}}$.
\end{proof}

\subsection{Bounds on some single  sums}
\label{sec:Sing}

Michel~\cite[Proposition~1.1]{Mich} gives a bound for the sum of the function
$\sym_n (\vartheta)$, given by~\eqref{eq:sym_n} twisted by
additive characters.

We refer to~\cite{IwKow} for background on characters. We use the notation
$\ep(z) = \exp(2 \pi i z/p)$ and record here the following immediate consequence
of ~\cite[Proposition~1.1]{Mich}.

\begin{lemma}
\label{lem:Mich bound0} If the polynomials $f(Z), g(Z) \in \Z[Z]$
satisfy~\eqref{eq:Nondeg}, for any prime $p$ we have
$$
\sum_{\substack{w \in \F_p\\
\Delta(w) \ne 0}}
\sym_n\(\psi_p(E(w))\) \ep\(mw\) \ll
np^{1/2},
$$
uniformly over  all
integers $m\ge 0$ and $n\ge 1$.
\end{lemma}

We need the following analogue of~\cite[Proposition~1.1]{Mich}
(in a more precise form than Lemma~\ref{lem:Mich bound0}) for the sum of the function
$\sym_n (\vartheta)$ twisted by
multiplicative characters.

\begin{lemma}
\label{lem:Mich bound}
Given a prime $p$, if the polynomials $f(Z), g(Z) \in \Z[Z]$
satisfy~\eqref{eq:Nondeg p}, then for any  multiplicative character
$\chi$ of $\F_p^*$ and any integer $n \ge 1$,  we have
$$
\left| \sum_{\substack{w \in \F_p\\
\Delta(w) \ne 0}}
\sym_n\(\psi_p(E(w))\) \chi\(w\) \right|
\le (n+1) \sqrt{p} \deg \Delta.
$$
\end{lemma}

\begin{proof}
The proof is similar to that of~\cite[Proposition~1.1]{Mich}
and is  a rather standard application of techniques of
\'etale cohomology (see, for example,~\cite{Milne} for a general reference)
and the work of Deligne and Katz (see~\cite{Katz1,Katz2}).
So, we only
point out the changes that need to be made to the argument
of the proof of~\cite[Proposition~1.1]{Mich} and refer
to~\cite{Mich} for more details and  the main argument.

We work over $\F_p$.
Let $U = \A^1 -\{Z\Delta(Z)=0\}$ and $\mathcal{E}$ the total space of
the family of elliptic curves given by~\eqref{eq:Family AB} over $U$.
Consider the map $\pi: \mathcal{E} \to U$ (given by $Z$).

Using  the standard notation of $\pi_!$ as the  direct image functor
and  $R^1\pi_!$ as the  first derived functor of $\pi_!$ (see~\cite{Milne}),
we now consider the sheaf $\cF = R^1 \pi_! \Q_{\ell}(1/2)$.

The desired result follows easily, for example,
from~\cite[Key Lemma, Page~286]{Katz1},
applied to $\Sym_n(\cF)\otimes \cL_{\chi}$,
once the hypotheses are checked, where $\cL_{\chi}$ is the Kummer
sheaf associated to $\chi$, as in~\cite{Katz2} and
that the needed facts about it are proved  in~\cite[Section~7]{Katz2}.

Michel~\cite{Mich} only needs to
work in the larger open set $\A^1 - \{ \Delta=0 \}$, but $\cL_{\chi}$ is not
well-behaved at $Z=0$ unlike the sheaf corresponding to an additive character.
On the other hand, $\cL_{\chi}$ is tamely ramified.
Just as in~\cite{Mich}, using  $\Sym_n$ the $n$-th symmetric power of a sheaf,
we obtain the triviality of the following cohomology groups
$$
H^i(U,\Sym_n(\cF)\otimes \cL_{\chi})=0, \qquad i\ne 1,
$$
because of the monodromy of $\cF$ computed there and the fact that
$\cL_{\chi}$  is a pure sheaf of rank one over $U$.

To complete the proof
we need a formula for the dimension of the first cohomology group
$H^1(U,\Sym_n(\cF)\otimes \cL_{\chi})$.
As both $\Sym_n(\cF)$ and  $\cL_{\chi}$ are  lisse over $U$ and $\cL_{\chi}$ is
tame of rank one, this dimension is the rank of $\Sym_n(\cF)$, namely $n+1$, times the Euler
characteristic of $U $, that is
$$\# (\P^1 - U) + 2 \cdot 0 - 2 \le \deg \Delta,
$$
proving the desired estimate.
\end{proof}

Using characters to detect a multiplicative  subgroup of $\F_p^*$, we immediately
derive from Lemma~\ref{lem:Mich bound} a more general result.

\begin{lemma}
\label{lem:Subgr} Given a prime $p$, if the polynomials $f(Z), g(Z) \in \Z[Z]$
satisfy~\eqref{eq:Nondeg p}, then for any multiplicative subgroup
$\cG \subseteq \F_p^*$, any  multiplicative character
$\chi$ of $\F_p^*$ and any integer $n \ge 1$,  we have
$$
\left| \sum_{\substack{w \in \cG \\
\Delta(w) \ne 0}}
\sym_n\(\psi_p(E(w))\) \chi(w) \right|
\le (n+1)\sqrt{p}\deg \Delta .
$$
\end{lemma}

\begin{proof}
Let $\cX_p$ denote the set of all $p-1$ multiplicative characters of $\F_p$.
Using the orthogonality of multiplicative characters, we obtain
\begin{equation*}
\begin{split}
 \sum_{\substack{w \in \cG \\
\Delta(w) \ne 0}} &
\sym_n\(\psi_p(E(w))\) \chi(w) \\
& \quad  = \sum_{\substack{u \in \F_p^* \\
\Delta(u) \ne 0}}
\sym_n\(\psi_p(E(u))\)
\sum_{w\in \cG} \frac{\chi(w)}{p-1}  \sum_{\phi \in \cX_p}\phi(wu^{-1}) \\
& \quad = \frac{1}{p-1} \sum_{\phi \in \cX_p} \sum_{\substack{u \in \F_p^* \\ \Delta(u) \ne 0}} \sym_n\(\psi_p(E(u))\) \bar{\phi} (u) \sum_{w\in \cG} \chi(w) \phi(w),
\end{split}
\end{equation*}
where $\bar{\phi} (u) = \phi (u^{-1})$.
So, Lemma~\ref{lem:Mich bound} yields that
\begin{align*}
\left| \sum_{\substack{w \in \cG \\
\Delta(w) \ne 0}}
\sym_n\(\psi_p(E(w))\) \chi(w) \right|
& \le \frac{(n+1)\sqrt{p}\deg \Delta }{p-1} \sum_{\phi \in \cX_p} \left| \sum_{w\in \cG} \chi(w)\phi(w) \right| \\
& = (n+1)\sqrt{p}\deg \Delta,
\end{align*}
where the identity follows from the fact that the sum $\sum_{w\in \cG} \chi(w)\phi(w)$ is equal to $\# \cG$ if the restriction of $\phi$ to $\cG$ is the inverse of $\chi$ and zero otherwise.
\end{proof}

From  Lemma~\ref{lem:Subgr}, we see that for any polynomials $f(Z), g(Z) \in \Z[Z]$
satisfying~\eqref{eq:Nondeg},  we have
$$
 \sum_{\substack{w \in \cG \\
\Delta(w) \ne 0}}
\sym_n\(\psi_p(E(w))\) \chi(w)
\ll n \sqrt{p}, 
$$
which is how we usually apply it.

Furthermore, we also have an analogue of Lemma~\ref{lem:Subgr}
for incomplete sums  which follows from the standard reduction
between complete and incomplete sums
(see~\cite[Section~12.2]{IwKow}). 

\begin{lemma}
\label{lem:IncompSubgr}
If the polynomials $f(Z), g(Z) \in \Z[Z]$
satisfy~\eqref{eq:Nondeg}, then for any prime $p$, any integer $\lambda$
with $\gcd(\lambda,p) =1$ and of multiplicative order $r$ modulo $p$,
 any positive integer $T \le r$ and  for any integer $n \ge 1$,  we have
$$
  \sum_{\substack{t =1 \\
\Delta(\lambda^t) \not\equiv 0 \pmod p}}^T
\sym_n\(\psi_p(E(\lambda^t))\)
\ll  n \sqrt{p} \log p.
$$
\end{lemma}

\begin{proof}
The proof is based on the standard reduction between complete and incomplete sums (see~\cite[Section~12.2]{IwKow}). Indeed, let
$\cG  \subseteq \F_p^*$ be the multiplicative subgroup
of $\F_p^*$ generated by $\lambda$.
Let $d=(p-1)/r$. Then, there exists a  primitive
element $\xi \in \F_p^*$ with $\lambda = \xi^d$.
Now for $w \in\F_p^*$ we denote by
$\ind w$ the unique integer  $z\in [0, p-2]$ with $w = \xi^z$.
Using the orthogonality
of exponential function $\e(z) = \exp(2 \pi i z)$, we write
\begin{equation*}
\begin{split}
& \sum_{\substack{t =1 \\
\Delta(\lambda^t) \not\equiv 0 \pmod p}}^T
\sym_n\(\psi_p(E(\lambda^t))\)\\
& \qquad \quad  = \sum_{\substack{w \in \F_p \\
\Delta(w) \ne 0}}
\sym_n\(\psi_p(E(w))\)
\sum_{t =1}^T \frac{1}{p-1}  \sum_{s=0}^{p-2}\e\( \frac{s(\ind w-dt)}{p-1}\).
\end{split}
\end{equation*}
Writing $\chi_s(w) = \e\(  s \ind w/(p-1)\)$
and changing the order of summation we obtain
\begin{equation*}
\begin{split}
& \sum_{\substack{t =1 \\
\Delta(\lambda^t) \not\equiv 0 \pmod p}}^T
\sym_n\(\psi_p(E(\lambda^t))\)\\
& \qquad \quad  =
 \frac{1}{p-1}  \sum_{s=0}^{p-2}\sum_{\substack{w \in \F_p \\
\Delta(w) \ne 0}}
\sym_n\(\psi_p(E(w))\) \chi_s(w)
\sum_{t =1}^T\e\( -st/r\).
\end{split}
\end{equation*}
It is easy to check that $\chi_s(w)$ is a multiplicative character
of $\F_p^*$ for any $0\le s \le p-2$. Thus by Lemma~\ref{lem:Mich bound},
$$
\sum_{\substack{t =1 \\
\Delta(\lambda^t) \not\equiv 0 \pmod p}}^T
\sym_n\(\psi_p(E(\lambda^t))\)  \ll
n   p^{-1/2}   \sum_{s=0}^{p-2}\left|
\sum_{t =1}^T\e\( st/r\)\right|.
$$
Using~\cite[Equation~(8.6)]{IwKow}, we know that if $r \nmid s$, we have
$$
\left|
\sum_{t =1}^T\e\( st/r\)\right| \le \frac{1}{2 \| s/r\|},
$$
where $\| s/r\|$ denotes the distance of $s/r$ to the nearest integer. The result now follows. 
\end{proof}

Finally we need the following slight generalisation of~\cite[Lemma~10]{Shp3}, which  is   based on
Lemma~\ref{lem:Mich bound0} and the same standard reduction
between complete and incomplete sums
(see~\cite[Section~12.2]{IwKow}) as we used in the proof of Lemma~\ref{lem:IncompSubgr}. 

\begin{lemma}
\label{lem:Interv}
If the polynomials $f(Z), g(Z) \in \Z[Z]$
satisfy~\eqref{eq:Nondeg}, then
for any prime $p$, any integers $M,N\ge 1$  and $k$ with $\gcd(k,p)=1$,   we have
$$
\sum_{\substack{m = M+1\\
\Delta(km) \not \equiv 0 \pmod p}}^{M+N}
\sym_n\(\psi_p(E(km))\)  \ll
n \(N p^{-1/2} + p^{1/2} \log p\)  ,
$$
uniformly over all integers $n\ge 1$.
\end{lemma}

\subsection{Bounds on some  bilinear sums}
\label{sec:Bilin}
The following bound of bilinear sums with ``weights'' is a direct
application of Lemma~\ref{lem:Mich bound}.

\goodbreak
\begin{lemma}
\label{lem:Bilin}
If the polynomials $f(Z), g(Z) \in \Z[Z]$
satisfy~\eqref{eq:Nondeg},
then for any prime $p$, any $U,V\ge 1$ and non-empty  sets of integers
$\cU \subseteq [1,U]$,  $\cV \subseteq [1,V]$ with
$\gcd(uv,p)=1$ for $u \in \cU$, $v \in \cV$,
and two sequences
of complex numbers
$\{\alpha_u\}_{u \in \cU}$ and  $\{\beta_v\}_{v \in \cV}$  with
$$
\max_{u \in \cU} |\alpha_u| = A \mand \max_{v \in \cV} |\beta_v| = B ,
$$
and for any integer $n \ge 1$,  we have
\begin{equation*}
\begin{split}
\ssum_{\substack{u\in \cU, v \in \cV\\
\Delta(uv) \not \equiv 0 \!\!\!\!\! \pmod p}}
\alpha_u \beta_v &\sym_n\(\psi_p(E(uv))\)\\
& \ll n AB  \sqrt{  \# \cU (U/p+1) \# \cV (V/p+1) p}.
\end{split}
\end{equation*}
\end{lemma}

\begin{proof}
Let $\cX_p$ denote the set of all $p-1$ multiplicative characters of $\F_p$. We denote by $S$ the sum to be bounded.
Using the orthogonality of multiplicative characters and the fact that $\chi(w^{-1}) = \overline \chi(w)$ for the complex conjugated
character $ \overline \chi$, we write
\begin{equation*}
\begin{split}
S&=\frac{1}{p-1} \sum_{\chi \in \cX_p}
  \sum_{\substack{w \in \F_p\\
\Delta(w) \ne 0}}
 \sym_n\(\psi_p(E(w))\)  \overline \chi(w)    \sum_{u\in \cU}
\alpha_u \chi(u)
\sum_{v \in \cV}  \beta_v \chi(v).
\end{split}
\end{equation*}
Using Lemma~\ref{lem:Mich bound} and the Cauchy inequality, we have
\begin{equation}
\label{eq:Michel}
\begin{split}S
& \ll n p^{-1/2} \sum_{\chi \in \cX_p} \left| \sum_{u\in \cU}
\alpha_u \chi(u) \right|
\left|\sum_{v \in \cV}  \beta_v \chi(v)\right|.
 \\&  \ll n p^{-1/2} \left(\sum_{\chi \in \cX_p} \left| \sum_{u\in \cU}
\alpha_u \chi(u) \right|^2  \sum_{\chi \in \cX_p}
\left|\sum_{v \in \cV}  \beta_v \chi(v)\right|^2\right)^{1/2}.
\end{split}
\end{equation}
Applying the orthogonality of multiplicative characters again,
we derive
\begin{equation*}
\begin{split}
 \sum_{\chi \in \cX_p} \left| \sum_{u\in \cU}
\alpha_u \chi(u) \right|^2
&   = \sum_{\chi \in \cX_p}  \sum_{u_1,u_2\in \cU}
\alpha_{u_1} \overline\alpha_{u_2} \chi(u_1 u_2^{-1})  \\
&  =  (p-1) \sum_{\substack{u_1,u_2\in \cU\\ u_1 \equiv u_2 \pmod p}}
\alpha_{u_1} \overline\alpha_{u_2}.
\end{split}
\end{equation*} 
Hence,
\begin{equation}
\label{eq:Ort1}
 \sum_{\chi \in \cX_p} \left| \sum_{u\in \cU}
\alpha_u \chi(u) \right|^2  \le  (p-1) A^2 \# \cU (U/p+1).
\end{equation}
Similarly, we have
\begin{equation}
\label{eq:Ort2}
\sum_{\chi \in \cX_p}
\left|\sum_{v \in \cV}  \beta_v \chi(v)\right|^2 \le (p-1) B^2 \# \cV (V/p+1).
\end{equation}
Substituting~\eqref{eq:Ort1} and~\eqref{eq:Ort2} into~\eqref{eq:Michel}, we conclude the proof.
\end{proof}

Finally, we need the following modification of Lemma~\ref{lem:Bilin}.

\begin{lemma}
\label{lem:BilinGen}
If the polynomials $f(Z), g(Z) \in \Z[Z]$
satisfy~\eqref{eq:Nondeg},
then for any prime $p$, any integers
$U,V,W\ge 1$ with $U\ge W$ and $V\ge 2$, two sequences of
integers $\{W_u\}_{u=W}^U$ and $\{V_u\}_{u=W}^U$
with $1 \le W_u \le  V_u \le V$ for each $u$ and
two sequences  of complex numbers
$\{\alpha_u\}_{u=W}^U$ and  $\{\beta_v\}_{v=1}^V$  with
$$
\max_{u=W, \ldots, U} |\alpha_u| = A \mand \max_{v =1, \ldots, V} |\beta_v| = B ,
$$
and for any integer $n \ge 1$,  we have
\begin{equation*}
\begin{split}
\sum_{u=W}^U  \alpha_u&\sum_{\substack{ v =W_u\\
\Delta(uv) \not \equiv 0 \pmod p}}^{V_u}
\beta_v  \sym_n\(\psi_p(E(uv))\)\\
& \ll n AB \sqrt{ V (U-W+1) (U/p+1) (V/p+1) p} \log V.
\end{split}
\end{equation*}
\end{lemma}

\begin{proof}
First note that
 for any $n \ge 1$, we have 
\begin{equation}
\label{eq:symn}
|\sym_n (\vartheta)| \leq n+1.
\end{equation}
Hence, the contribution from the terms with $p\mid uv$ is at most
$$
2(n+1)ABV(1+(U-W+1)/p),
$$
which is not greater than the desired upper bound. Thus, we can assume that $\alpha_u = 0$ if $p\mid u$, and
$\beta_v = 0$ if $p\mid v$.

We define  $\eV(z) = \exp(2 \pi i z/V)$. Then, for each
inner sum, using the orthogonality of exponential
functions,  we write
\begin{equation*}
\begin{split}
\sum_{\substack{v =W_u\\
\Delta(uv) \not \equiv 0 \pmod p}}^{V_u}&
\beta_v  \sym_n\(\psi_p(E(uv))\)\\
&=\sum_{\substack{ v =1\\
\Delta(uv) \not \equiv 0 \pmod p}}^V  \beta_v  \sym_n\(\psi_p(E(uv))\)\\
& \qquad \qquad \quad \sum_{w =W_u}^{V_u}  \frac{1}{V} \sum_{-V/2<s\le V/2}
\eV(s(v-w))\\
&=\frac{1}{V} \sum_{-V/2<s\le V/2}
\sum_{w =W_u}^{V_u}  \eV(-sw)\\
& \qquad \qquad \quad
\sum_{\substack{ v =1\\\Delta(uv) \not \equiv 0 \pmod p}}^V
\beta_v \eV(sv) \sym_n\(\psi_p(E(uv))\) .
\end{split}
\end{equation*}
In view of~\cite[Equation~(8.6)]{IwKow}, for each $u =1, \ldots, U$
and every integer $s$ such that $|s|\le V/2$ we can write
$$
\sum_{w =W_u}^{V_u}  \eV(-sw)
=\sum_{w =1}^{V_u}  \eV(-sw)
-\sum_{w =1}^{W_u-1}  \eV(-sw)
=\eta_{s,u}\frac{V}{|s|+1}
$$
for some complex number $\eta_{s,u}\ll 1$.  Thus, if we put
$\widetilde\alpha_{s,u}=\alpha_u \eta_{s,u}$ and
$\widetilde\beta_{s,v}=\beta_v \eV(sv)$,
it follows that
\begin{equation*}
\begin{split}
\sum_{u=1}^U & \alpha_u\sum_{\substack{ v =W_u\\
\Delta(uv) \not \equiv 0 \pmod p}}^{V_u}
\beta_v  \sym_n\(\psi_p(E(uv))\)\\
& = \sum_{-V/2<s\le V/2}\frac{1}{|s|+1}
\sum_{u=W}^U  \sum_{\substack{ v = 1\\
\Delta(uv) \not \equiv 0 \pmod p}}^{V }
\widetilde \alpha_{s,u} \widetilde \beta_{s,v} \sym_n\(\psi_p(E(uv))\).
\end{split}
\end{equation*}
Applying Lemma~\ref{lem:Bilin} with the sequences
$(\widetilde\alpha_{s,u})_{u=W}^U$ and $(\widetilde\beta_{s,v})_{v=1}^V$ for each $s$,
and noting that
$$
\sum_{-V/2<s\le V/2}\frac{1}{|s|+1}\ll\log V
$$
we derive the desired upper bound.
\end{proof}

We are now ready to establish our main technical tool,  which
gives a bound of bilinear sums   over a certain ``hyperbolic''
region of summation.

\begin{lemma}
\label{lem:BilinHyperb}
If the polynomials $f(Z), g(Z) \in \Z[Z]$
satisfy~\eqref{eq:Nondeg},
then for any prime $p$, any integers
$U,V, W\ge 1$, a sequence  of
integers $\{Z_u\}_{u=1}^U$
with
$$
1 \le W \le U, \quad U\ge 2 \mand  1 \le Z_u < V, \ u = W, \ldots,  U,
$$
and two sequences
of complex numbers
$\{\alpha_u\}_{u=W}^U$ and  $\{\beta_v\}_{v=1}^V$  with
$$
\max_{u=W, \ldots, U} |\alpha_u| = A \mand \max_{v =1, \ldots, V} |\beta_v| = B ,
$$
and for any integer $n \ge 1$,  we have
\begin{equation*}
\begin{split}
\sum_{W \le u \le U} &  \alpha_u  \sum_{\substack{Z_u\le v \le V/u \\
\Delta(uv) \not \equiv 0 \pmod p}}
\beta_v  \sym_n\(\psi_p(E(uv)\)\\
& \ll
n AB \(V p^{-1/2} + V W^{-1/2} +  (U  V)^{1/2}  + (V  p)^{1/2}\) \log U \log V.
\end{split}
\end{equation*}
\end{lemma}

\begin{proof}
Note that the desired upper bound is better than the direct consequence of Lemma~\ref{lem:BilinGen}.

Let $\alpha_u=0$ if $u < W$ or $u>U$. We also set
$ I = \fl{\log W} $  and $ J = \fl{\log U}
$,
and consider the half-open intervals
$$
\cI_j = [\exp(j), \exp(j+1)), \qquad ( I\le j\le J).
$$
Then,
\begin{equation*}
\begin{split}
\sum_{W \le u \le U}  \alpha_u&\sum_{\substack{Z_u\le v \le V/u \\
\Delta(uv) \not \equiv 0 \pmod p}}
\beta_v  \sym_n\(\psi_p(E(uv))\)
\\ & \qquad =
\sum_{j=I}^J\sum_{u \in \cI_j}
\sum_{\substack{Z_u\le v \le V/u \\
\Delta(uv) \not \equiv 0 \pmod p}}
\alpha_u \beta_v \sym_n\(\psi_p(E(uv))\).
\end{split}
\end{equation*}
Using Lemma~\ref{lem:BilinGen}, each inner double sum
satisfies the bound
\begin{equation*}
\begin{split}
\sum_{u \in \cI_j}&
\sum_{\substack{Z_u\le v \le V/u \\
\Delta(uv) \not \equiv 0 \pmod p}}
\alpha_u \beta_v \sym_n\(\psi_p(E(uv))\)\\
&\quad\ll n AB
\sqrt{ V  (\exp(j)/p+1) (V\exp(-j)/p+1) p} \log V\\
&\quad\ll n AB   \sqrt{V^2p^{-1} + V^2\exp(-j) + V  \exp(j)  + V  p}\log V.
\end{split}
\end{equation*}
Summing over $j \in [I, J]$,  we conclude the proof.
\end{proof}

\subsection{Vaughan's Identity}
\label{sec:Vaugh}

As usual, we use $\mu(d)$  to denote the M{\"o}bius function and
$\Lambda$  to denote the von~Mangoldt function given by
$$
\Lambda(t)=
\begin{cases}
\log \ell &\quad\text{if $t$ is a power of some prime $\ell$,} \\
0&\quad\text{if $t$ is not a prime power.}
\end{cases}
$$

We need the following result of Vaughan~\cite{Vaughan,Vau}, which is stated
here in the form given in~\cite[Chapter~24]{Dav} (see also~\cite[Section~13.4]{IwKow}).

\begin{lemma}
\label{lem:Vau}
For any complex-valued function $\psi(t)$ and any real numbers $K,M\ge 1$
with $KM\le L$ and $L\ge 2$, we have
$$
\sum_{t=1}^L\Lambda(t)\psi(t)\ll\Sigma_1+\Sigma_2 \log (KM)+\Sigma_3 \log L+\Sigma_4,
$$
where
\begin{equation*}
\begin{split}
\Sigma_1&=\left|\,\sum_{t\le M}\Lambda(t)\psi(t)\right|,\\
\Sigma_2&=\sum_{k \le KM}\left|\,\sum_{m\le L/k}\psi(k
m)\right| ,\\
\Sigma_3&=\sum_{k\le K}\,\max_{w\ge
1}\Bigg| \sum_{w\le m\le L/k}\psi(k m)\Bigg|,\\
\Sigma_4&=\Bigg| \sum_{M<m \le L/K}
\Lambda(m) \sum_{K<k \le L/m}
\Bigg( \sum_{\substack{d\,\mid\,k\\d\le
K}}\mu(d) \Bigg) \psi(km)\Bigg|.
\end{split}
\end{equation*}
\end{lemma}

So, Lemma~\ref{lem:Vau} reduces the problem of estimating the
sums over primes to sums over consecutive integers
and bilinear sums, which for the function $\sym_n$ are available
from Sections~\ref{sec:Sing} and~\ref{sec:Bilin}.

\subsection{Bounds on some sums over primes}
\label{sec:Prime}

We start with showing that~\cite[Theorem~1.5]{FKM}  applies to the functions
$\sym_n\(\psi_p(E(t))\)$ for every $n \ge 1$ and leads to the following 
bound:

\begin{lemma}
\label{lem:Prime1}
For any fixed prime $p$ and $0<\eta < 1/48$,
and  polynomials $f(Z), g(Z) \in \Z[Z]$ that 
satisfy~\eqref{eq:Nondeg p}, 
we have
$$ \sum_{\substack{\textrm{prime } \ell \le L\\
\Delta(\ell) \not\equiv 0\!\!\!\!\! \pmod p}}
\sym_n\(\psi_p(E(\ell))\)    \ll n^A  \pi(L) \( 1  + p/L\)^{1/12} p^{-\eta} 
$$
for any integer $n \ge 1$, 
where the implied constant and the constant $A \ge 1$  depend only on
$f(Z)$, $g(Z)$ and $\eta$. 
\end{lemma}

\begin{proof}
We first remark that in the result of \cite[Equation~(1.3) of Theorem~1.5]{FKM}, the factor $X$ can be replaced by $\pi(X)$. This comes from the bound in \cite[page 1714]{FKM}
$$
{\mathcal S}_{V,X}(\Lambda, K)\ll (pQ)^{\varepsilon} QXp^{-\eta}
$$
and an integration by parts. 

Now, we wish to apply~\cite[Equation~(1.3) of Theorem~1.5]{FKM} to the trace weight
$K(\ell) = \sym_n\(\psi_p(E(\ell))\)$ associated to
the sheaf $\Sym_n(\cF)$ considered in the proof of Lemma~\ref{lem:Mich bound} 
(and in~\cite{Mich}),
where $\eta$ here corresponds to $\eta/2$ in \cite{FKM}. We need to verify that this
sheaf is not exceptional in the sense of \cite{FKM} and estimate its conductor.
The sheaf is not exceptional because its monodromy is not abelian, as seen in the
proof of~\cite[Lemma~3.1]{Mich}. The conductor is the dimension of the $H^1$
which has been estimated in the proof of Lemma~\ref{lem:Mich bound} and is linear in $n$.
The result now follows.
\end{proof}

\begin{rem}
{\rm
Although this is not explicitly stated  in~\cite{FKM}, the constant 
$A$ seems to be absolute (and in fact of very moderate value). If it is worked out 
explicitly then Theorem~\ref{thm:S-T Primes 1} can also be made more explicit. 
}
\end{rem}

We now use a different method to bound the sums in Lemma \ref{lem:Prime1} which is more efficient for $L \ge p$. 
First we estimate the sums weighted by the  von~Mangoldt function.

\begin{lemma}
\label{lem:Lambda}
 If the polynomials $f(Z), g(Z) \in \Z[Z]$
satisfy~\eqref{eq:Nondeg}, then for any prime $p$, and any integers $n \ge 1, L\ge 2$,
we have
\begin{equation*}
\begin{split}
 \sum_{\substack{1 \le t \le L\\
\Delta(t) \not \equiv 0 \pmod p}} \Lambda(t) &
\sym_n\(\psi_p(E(t))\)  \\
& \ll n  \( L p^{-1/2} + L^{5/6}  + L^{1/2} p^{1/2} \)   L^{c/\log\log L},
\end{split}
\end{equation*}
for some absolute constant $c>0$.
\end{lemma}

\begin{proof} 
We remark that the trivial upper bound on the above sum is $O(nL)$,
so we can assume  $ p\le L$.
First, for any integer $t$ put $\delta(t)=1$ if $\Delta(t) \ne 0 \pmod p$, and let $\delta(t)=0$ otherwise.
We fix some real numbers $K,M\ge 1$
with $KM\le L$.
We now need to estimate the sums $\Sigma_i$, $i =1, \ldots, 4$,
of Lemma~\ref{lem:Vau} with 
$$
\psi(t)=\sym_n\(\psi_p(E(t))\)\delta(t).
$$

By the prime number theorem, and using~\eqref{eq:symn} we can estimate $\Sigma_1$ trivially as
\begin{equation}
\label{eq:S1}
\Sigma_1 \ll n M.
\end{equation}

To estimate $\Sigma_2$, we choose another parameter $R$ and write
\begin{equation}
\label{eq:S2 S21 S22}
\Sigma_2=\Sigma_{2,1} + \Sigma_{2,2},
\end{equation}
where
\begin{equation*}
\Sigma_{2,1} =\sum_{k \le R}\left| \sum_{m\le L/k} \psi(km)\right|, \quad
\Sigma_{2,2} =\sum_{R < k \le KM}\left| \sum_{m\le L/k} \psi(km)\right|.
\end{equation*}

 For $\Sigma_{2,1}$
we estimate the inner sum
by Lemma~\ref{lem:Interv}
 for $k$ not divisible by $p$ and estimate the inner sum trivially for
other $k$. Hence, we obtain
\begin{equation}
\label{eq:S21}
\begin{split}
\Sigma_{2,1}&\ll n   \sum_{\substack{k \le R \\ p\nmid k}}  \(\frac{L}{kp^{1/2}} + p^{1/2} \log p\)
+ n\sum_{\substack{k \le R\\ p\mid k}} \frac{L}{k}  \\
&\ll n   \(\frac{L \log R}{p^{1/2}} + R p^{1/2} \log p \).
\end{split}
\end{equation}

For $\Sigma_{2,2}$ we apply Lemma~\ref{lem:BilinHyperb} (with $\beta_v=1$ and $\alpha_u=\pm 1$ according to the sign of the inner sum) and derive
\begin{equation}
\label{eq:S22} 
\Sigma_{2,2}  \ll n \(Lp^{-1/2}+ L R^{-1/2} +  (KLM)^{1/2}  +(L p)^{1/2} \) (\log L)^2. 
\end{equation}

We now choose $R = L^{2/3} p^{-1/3}$ and substitute the bounds~\eqref{eq:S21} and~\eqref{eq:S22}
into~\eqref{eq:S2 S21 S22}.
 Furthermore, we also note that we can write
$$
\Sigma_3=\sum_{k\le K}\left| \sum_{w_k\le m\le L/k}\psi(k m)\right|,
$$
where $w_k$, $1 \le k \le K$, are chosen to satisfy
$$
\left| \sum_{w_k\le m\le L/k}\psi(k m)\right|= \max_{w\ge
1}\left| \sum_{w\le m\le L/k}\psi(k m)\right|.
$$
Hence, $\Sigma_3$ can be split into two sums as in~\eqref{eq:S2 S21 S22}, and analogues of the
 bounds~\eqref{eq:S21} and~\eqref{eq:S22}
 apply to $\Sigma_3$ (with $K$ in place of $KM$).
 We also note that for  $L\ge p$ we have  $\log p \le \log L $ and  $L^{2/3} p^{1/6}  \le L^{5/6}$. 
Therefore,   we obtain
\begin{equation}
\label{eq:S2 S3}
\Sigma_2+ \Sigma_3 \ll  n  \( L p^{-1/2}  +L^{5/6}
  + (KLM)^{1/2}  +(L p)^{1/2} \) (\log L)^2 .
\end{equation}

In addition, notice that in the area of the summation in $\Sigma_4$ we always have
$k \le L/M$ and $m \le L/K$. 
We also observe that  the classical bound on
the divisor function $\tau(k)$ for $k\le L$, see~\cite[Theorem~13.12]{Apostol}, yields
$$
\sum_{\substack{d\,\mid\,k\\d\le
K}}\mu(d) \ll  \tau(k)   \le L^{c_1/\log\log L},
$$
where $c_1$ is some absolute constant. 
Hence, applying
Lemma~\ref{lem:BilinHyperb} to estimate $\Sigma_4$,  we deduce
\begin{equation}
\label{eq:S4} 
\Sigma_4
 \ll n \( Lp^{-1/2} + LM^{-1/2} + L K^{-1/2} + L^{1/2}p^{1/2} \) L^{c/\log\log L} 
\end{equation} 
for some absolute constant $c>0$. 

Comparing~\eqref{eq:S2 S3} with~\eqref{eq:S4}, we choose
$K=M= L^{1/3}$ to balance these estimates, which also dominate~\eqref{eq:S1}.
Then, substituting the
above estimates into Lemma~\ref{lem:Vau} we obtain
\begin{equation*}
\begin{split}
 \sum_{\substack{1 \le t \le L\\
\Delta(t) \not \equiv 0\!\!\!\!\! \pmod p}} & \Lambda(t)
\sym_n\(\psi_p(E(t))\)  \\
& \ll n  \( L p^{-1/2} + 
L^{5/6}  + L^{1/2} p^{1/2} \) L^{c/\log\log L}.
\end{split}
\end{equation*}  
This completes the proof.
\end{proof}

Via partial summation and using Lemma \ref{lem:Lambda},  we  are now immediately ready to obtain  our main technical ingredient for the proof of Theorem~\ref{thm:Prime2}.

\begin{cor}
\label{cor:Prime2}
If the polynomials $f(Z), g(Z) \in \Z[Z]$
satisfy~\eqref{eq:Nondeg}, then for any prime $p$, and any integer $n \ge 1$,  we have
\begin{equation*}
\begin{split}
& \sum_{\substack{\textrm{prime } \ell \le L\\
\Delta(\ell) \not\equiv 0\!\!\!\!\! \pmod p}}
\sym_n\(\psi_p(E(\ell))\)  \\
& \qquad\qquad \ll n  \( L p^{-1/2} +
L^{5/6}  +(L p)^{1/2} \)   L^{c/\log\log L},
\end{split}
\end{equation*}
where $c>0$ is some absolute constant.
\end{cor}

\subsection{Distribution of multiplicative orders}
\label{sec:MultOrder}

For integer $\lambda$  and prime $p$ with $\gcd(p,\lambda)=1$, let $\ord_p \lambda$ denote the multiplicative order of $\lambda$ modulo $p$. Then  for any real $\alpha\in (0,2)$, define
$$
S_\alpha(x;\lambda)=\sum_{\substack{p\leq x \\ \gcd(p, \lambda)=1}} \frac{1}{(\ord_p\lambda)^\alpha}.
$$
It follows from
\cite[Corollary~5]{IndTim} that
(taking $r=1$ there) 
\begin{equation}
\label{eq:IndTim}
S_1(x;\lambda)
\ll x^{1/2}\frac{ (\log \log x)^{1+\delta}}{(\log x)^{1+\delta/2}},
\end{equation}
where $\delta$ is given by~\eqref{eq:Erdos}.
 
 Let 
 $$
H_{\cP}(x,y,2y)=\# \{ p\le x~:~\exists \, d\in (y,2y], \, d\mid p-1\}.
$$
We first need the following  consequence of a result of Ford~\cite[Theorem~6 and Corollary~2]{F08}
(we   remark that the extension to $y \in [x^{1/2},  x^{3/4}]$ comes from the symmetry of the 
divisors $d$ and $(p-1)/d$).

\begin{lemma}
\label{lem:Ford}
For any $x\ge 2$ and $3\leq y\leq  x^{3/4}$, we have
$$H_{\cP}(x,y,2y)\ll  \frac{x}{ (\log x)(\log y)^\delta(\log\log y)^{3/2}}.$$
\end{lemma}

\begin{lemma}
\label{lem:MultOrd}
For any  integer $\lambda$ with $|\lambda|>1$, any real $\alpha\in (0,2)$ and $x \ge 3$, we have
$$
S_\alpha(x;\lambda)\ll_\alpha \frac{x^{1-\alpha/2} \log |\lambda|}{(\log x)^{1+(2-\alpha)\delta/2}(\log\log x)^{3(2-\alpha)/4}}.
$$
\end{lemma}

\begin{proof}
Let $3\leq y < z \leq  x^{3/4}$. We divide $S_\alpha(x;\lambda)$ into three parts $S_1$, $S_2$ and $S_3$   with $S_1$ corresponding to $\ord_p \lambda \leq y $, $S_2$ to $\ord_p \lambda \in (y,z]$ and
$S_3$ to $\ord_p \lambda >z$. As usual, let $\omega(n)$ be the number of distinct prime factors of integer $n\ne 0$.
Using the bound $\omega(n)\ll (\log n)/(\log\log n)$, we obtain 
$$
S_1 \leq \sum_{m\leq y} \frac{\omega(\lambda^m-1)}{m^\alpha} \ll_\alpha \frac{y^{2-\alpha}}{\log y} \log |\lambda|.
$$
Note that ${\rm ord}_p(\lambda) \mid p-1$. Then applying Lemma~\ref{lem:Ford}, we get
\begin{align*}
S_2  & \le \sum_{\substack{k\geq 0 \\ 2^{k}y\leq z}}\frac{1}{2^{k\alpha} y^\alpha}
H_{\cP}(x,2^{k}y,2^{k+1}y) \\
& \ll_\alpha \frac{x}{(\log x)y^\alpha(\log y)^\delta(\log\log y)^{3/2}}.
\end{align*}
In addition, we trivially have
$$S_3\leq \frac{\pi(x)}{  z^\alpha}.$$

Then taking $z=x^{3/4}$ and $y= \sqrt{x} (\log x)^{-\delta/2}(\log\log x)^{-3/4}$, we have
$$S_1+S_2+S_3\ll_\alpha \frac{x^{1-\alpha/2}\log |\lambda|}{(\log x)^{1+(2-\alpha)\delta/2}(\log\log x)^{3(2-\alpha)/4}},
$$
which completes the proof.
\end{proof}

\section{Proofs of Main Results}

\subsection{Proof of Theorem~\ref{thm:Subgroup}}

Let $\cH = \{w\in \cG: \Delta(w) \ne 0 \}$.
Applying Lemma~\ref{lem:Subgr}, we immediately have
$$
 \left| \sum_{w \in \cH} \sym_n(\psi_p(E(w))) \right| \ll n p^{1/2},
$$
which, together with Corollary~\ref{cor:ST Discrep}, yields
$$
N_p(\alpha,\beta;\cG) = \mu_{\tt ST}(\alpha,\beta) \# \cH
+ O\(\sqrt{p^{1/2} \#\cH}\).
$$
We complete the proof by noticing that
$$
 \# \cG - \deg \Delta \le \# \cH \le \# \cG.
$$

\subsection{Proof of Theorem~\ref{thm:ProductSet}}

Let
$$
\cH = \{ (u,v): u\in \cU, v\in \cV, \Delta(uv) \ne 0 \}.
$$
Using Lemma~\ref{lem:Bilin} (with $\alpha_u =\beta_v = 1$ and $U=V = p-1$), we obtain
$$
 \left| \sum_{(u,v) \in \cH} \sym_n(\psi_p(E(uv))) \right| \ll n \sqrt{\# \cU \# \cV p},
$$
which, combining with Corollary~\ref{cor:ST Discrep}, gives
$$
N_p(\alpha,\beta;\cU, \cV) = \mu_{\tt ST}(\alpha,\beta) \# \cH
+ O\((\# \cU \# \cV p)^{1/4} \sqrt{\#\cH}\).
$$
We conclude the proof by noticing that
$$
 \# \cU \# \cV - \min\{ \#\cU, \#\cV\}\deg \Delta \le \# \cH \le \# \cU \# \cV.
$$

\subsection{Proof of Theorems~\ref{thm:Prime1} and ~\ref{thm:Prime2}}
\label{sec:proof prime}

Let $\cL$ be the set of primes $\ell \le L$ such that $\Delta(\ell) \not\equiv 0 \pmod p$.

Without loss of generality we can assume that $\varepsilon \le 1/4$.
We  set 
$$
\eta = 1/48 - \varepsilon/24.
$$
Applying Lemma~\ref{lem:Prime1}, we derive
$$
 \left| \sum_{\ell \in \cL} \sym_n(\psi_p(E(\ell))) \right|
 \ll  n^A  \pi(L) \( 1  + p/L\)^{1/12} p^{-\eta}.
$$
Now, combining this with Corollary~\ref{cor:ST Discrep},  and using that $\# \cL \le \pi(L)$ and 
$p/L < p^{1/4-\varepsilon}$,  we obtain
\begin{equation*}
\begin{split}
Q_p(\alpha,\beta;L) &-  \mu_{\tt ST}(\alpha,\beta) \# \cL\\
&  \ll \pi(L)^{A/(A+1)}\(\pi(L) \( 1  + p/L\)^{1/12} p^{-\eta}\)^{1/(A+1)} \\
&  = \pi(L)\( \( 1  + p/L\)^{1/12} p^{-\eta}\)^{1/(A+1)} \\
&  \ll \pi(L)\( \( 1  + p^{1/4-\varepsilon}\)^{1/12} p^{-\eta}\)^{1/(A+1)} \\
& \ll \pi(L) p^{(1/48 -\varepsilon/12 -\eta)/(A+1)}  =  \pi(L) p^{-\rho},
\end{split}
\end{equation*}
where  
$$
\rho=\frac{\varepsilon}{24(A+1)} > 0
$$
(provided that $0 < \varepsilon \le 1/4$).

This  completes the proof of Theorems~\ref{thm:Prime1} by noticing that 
$$
\pi(L) -  \frac{c_1\pi(L)}{p-1}\deg \Delta \le \# \cL \le \pi(L). 
$$
where $c_1>0$ is some absolute constant according to the Dirichlet theorem on 
primes in arithmetic progressions. 

The proof of Theorem~\ref{thm:Prime2} is identical 
to that of  Theorem~\ref{thm:Prime1}, 
except that 
we use   Corollary~\ref{cor:Prime2} instead of Lemma~\ref{lem:Prime1}.

\subsection{Proof of Theorem~\ref{thm:S-T Setprod}}
We consider slightly more general 
 settings, when $\cU, \cV \subseteq [1,T]$
for some positive integer $T\le x$, because some of the intermediate bounds
can be of further use.

For any prime $p$, let
$$
\cD_p = \{ (u,v)~:~u\in \cU, \, v\in \cV, \, uv \equiv 0 \pmod p \},
$$
and 
$$
\cH_p = \{ (u,v)~:~u\in \cU,\,  v\in \cV, \, \Delta(uv) \not\equiv 0 \pmod p \}.
$$
We denote by $M_p(\alpha,\beta;\cU,\cV)$ the number of pairs $(u,v)\in \cH_p$ such that  $\psi_p(E(uv)) \in [\alpha,\beta]$. Without loss of generality, we assume that
$$
\# \cU \ge \# \cV.
$$ 
It follows from Lemma~\ref{lem:Bilin} (with $\alpha_u =\beta_v = 1$ and $U=V = T$)
and~\eqref{eq:symn} that
\begin{align*}
 & \left| \sum_{(u,v) \in \cH_p} \sym_n(\psi_p(E(uv))) \right| \\
 & \qquad \leq \left| \sum_{\substack{(u,v) \in \cH_p\\ \gcd(uv,p)=1}} \sym_n(\psi_p(E(uv))) \right| + \left| \sum_{\substack{(u,v) \in \cH_p\\ \gcd(uv,p)\ne 1}} \sym_n(\psi_p(E(uv))) \right| \\
 &\qquad \ll n (T/p+1) p^{1/2} (\# \cU \# \cV )^{1/2} + n (T/p )\# \cU. 
\end{align*}  
So, using Corollary~\ref{cor:ST Discrep}, we have
\begin{equation*}
\begin{split}
M_p(&\alpha,\beta;\cU, \cV)  -  \mu_{\tt ST}( \alpha,\beta)  \# \cH_p \\ &\ll (\#\cH_p)^{1/2}\Big( (T/p+1)^{1/2}  p^{1/4}(\# \cU \# \cV )^{1/4} +(T/p )^{1/2}(\# \cU)^{1/2}  \Big).
\end{split}
\end{equation*}
Noticing 
$$
 \# \cU \# \cV - \#\cD_p - \#\cV(T/p + 1)\deg \Delta  \le \# \cH_p \le \# \cU \# \cV,
$$
we obtain  
\begin{equation}
\begin{split}
\label{eq:M_p 1}
& M_p(\alpha,\beta;\cU, \cV)  - \mu_{\tt ST}(\alpha,\beta)\# \cU \# \cV \\
& \qquad \qquad \ll \#\cD_p + Tp^{-1}\# \cV + (T^{1/2}p^{-1/4}+p^{1/4}) (\# \cU \# \cV )^{3/4}\\
&\qquad\qquad  \qquad  \qquad \qquad \qquad \qquad  \qquad+ T^{1/2} p^{-1/2}  \# \cU (\# \cV)^{1/2} .
\end{split}
\end{equation}

Besides, it is easy to see that
\begin{align*}
\sum_{\substack{u \in \cU, v\in \cV \\ \Delta(uv) \ne 0}}
\pi_{E(uv)}(\alpha,\beta; x)
& = \sum_{\substack{u \in \cU, v\in \cV \\ \Delta(uv) \ne 0}}
\sum_{\substack{p\le x \\ \Delta(uv) \not\equiv 0  \pmod p \\ \psi_p(E(uv)) \in [\alpha,\beta]}} 1 \\
& = \sum_{p\le x} \sum_{\substack{u \in \cU, v\in \cV  \\ \Delta(uv) \not\equiv 0  \pmod p \\ \psi_p(E(uv)) \in [\alpha,\beta]}} 1
 = \sum_{p\le x} M_p(\alpha,\beta;\cU, \cV).
\end{align*}
Moreover, we estimate the sum of $\#\cD_p$ as follows: 
\begin{align*}
\sum_{p \le x} \#\cD_p \le \#\cU \sum_{v\in \cV} \omega(v) +
\#\cV \sum_{u\in \cU} \omega(u) \ll \#\cU\#\cV \log x,
\end{align*}
where, as before,  $\omega(w)$ is the number of distinct prime factors of integer $w \ne 0$.

Thus, applying~\eqref{eq:M_p 1} we deduce that
\begin{equation}
\begin{split}
\label{eq:M_p 2}
& \sum_{\substack{u \in \cU, v\in \cV \\ \Delta(uv) \ne 0}}
\pi_{E(uv)}(\alpha,\beta; x) -   \mu_{\tt ST}(\alpha,\beta)\pi(x)\# \cU \# \cV \\
& \qquad\ll  \sum_{p\le x}\Bigl( \#\cD_p + Tp^{-1}\# \cV + (T^{1/2}p^{-1/4}+p^{1/4}) (\# \cU \# \cV )^{3/4}\\
&\qquad\qquad\qquad\qquad\qquad\qquad\qquad\qquad+T^{1/2} p^{-1/2}\# \cU (\# \cV)^{1/2}  \Bigr) \\
&  \qquad  \ll \#\cU\#\cV \log x + T\# \cV \log x+ \pi(x) T^{1/2} x^{-1/2}\# \cU (\# \cV)^{1/2} \\
&  \qquad\qquad\qquad\qquad\qquad+ \pi(x)\(T^{1/2}x^{-1/4}+x^{1/4}\) \(\# \cU \# \cV \)^{3/4} .
\end{split}
\end{equation}

Now, to finish the proof, we substitute $T =x$ into~\eqref{eq:M_p 2} and notice that 
$T\# \cV \log x\leq x\(\# \cU \# \cV \)^{1/2}\log x \ll \pi(x) x^{1/4}  \(\# \cU \# \cV \)^{3/4} $.

\subsection{Proof of Theorem~\ref{thm:S-T GeomProgr}}
\label{sec:S-T GeomProgr}

For any prime $p$, let
$$
\cH_p = \{ t: 1\le t \le T, \, \Delta(\lambda^t) \not\equiv 0 \pmod p \}.
$$
We denote by $M_p(\alpha,\beta;T)$ the number of integers $t\in \cH_p$ such that  $\psi_p(E(\lambda^t)) \in [\alpha,\beta]$.
If $p\nmid \lambda$, we write $T= k_p \ord_p \lambda + s_p$ with $0\le s_p < \ord_p \lambda$.
Using Lemma~\ref{lem:Subgr} and Lemma~\ref{lem:IncompSubgr}, for $p\nmid \lambda$ we obtain
$$
 \left| \sum_{t \in \cH_p} \sym_n(\psi_p(E(\lambda^t))) \right| \ll n\(k_p \sqrt{p} +  \sqrt{p}\log p\).
$$
Hence, using Corollary~\ref{cor:ST Discrep}, we have
\begin{equation*}
M_p(\alpha,\beta;T) = \mu_{\tt ST}( \alpha,\beta)  \# \cH_p
 + O\(\sqrt{(k_p\sqrt{p} + \sqrt{p}\log p)\#\cH_p}\).
\end{equation*}
Noticing
$$
 T- (k_p +1)\deg \Delta \le \# \cH_p \le  T,
$$
for $p\nmid \lambda$ we have
\begin{equation}
\label{eq:M_pt}
M_p(\alpha,\beta;T) - \mu_{\tt ST}(\alpha,\beta)T
\ll \(k_p^{1/2}p^{1/4}+p^{1/4}(\log p)^{1/2}\) T^{1/2}.
\end{equation}
In addition, it is easy to see that
\begin{align*}
\sum_{\substack{1\le t \le T \\ \Delta(\lambda^t) \ne 0}}
\pi_{\lambda}(\alpha,\beta; t,x)
& = \sum_{\substack{1\le t \le T \\ \Delta(\lambda^t) \ne 0}}
\sum_{\substack{p\le x \\ \Delta(\lambda^t) \not\equiv 0  \pmod p \\ \psi_p(E(\lambda^t)) \in [\alpha,\beta]}} 1 \\
& = \sum_{p\le x}
\sum_{\substack{1\le t \le T  \\ \Delta(\lambda^t) \not\equiv 0 \pmod p \\ \psi_p(E(\lambda^t)) \in [\alpha,\beta]}} 1
= \sum_{p\le x } M_p(\alpha,\beta;T).
\end{align*}

For $p\mid \lambda$, we use the trivial bound $M_p(\alpha,\beta;T) \le T$.
 Thus, recalling~\eqref{eq:M_pt} and using $k_p \le T/\ord_p\lambda$,
we deduce that
\begin{align*}
&  \sum_{\substack{1\le t \le T \\ \Delta(\lambda^t) \ne 0}}
\pi_{\lambda}(\alpha,\beta; t,x)  -  \mu_{\tt ST}(\alpha,\beta)\pi(x)T \\
& \quad \ll T\log |\lambda| + \sum_{\substack{p\le x \\ p \nmid \lambda}} \(k_p^{1/2}p^{1/4}+p^{1/4}(\log p)^{1/2} \) T^{1/2} \\
& \quad \ll_{\lambda}  Tx^{1/4}S_{1/2}(x; \lambda) +   T^{1/2}x^{1/4}(\log x)^{1/2}\pi(x) \\
& \quad   \ll_{\lambda}   \frac{Tx}{(\log x)^{1+3\delta/4}(\log \log x)^{9/8}}  + T^{1/2}x^{1/4}(\log x)^{1/2}\pi(x),
\end{align*}
where the last inequality follows from Lemma~\ref{lem:MultOrd}. 

Recalling the condition  $T \ge   x^{1/2}
(\log x)^{1+3\delta/2}(\log \log x)^{ 9/4}$, we complete the proof.

\subsection{Proof of Theorems~\ref{thm:S-T Primes 1} and \ref{thm:S-T Primes 2}}

The results follow immediately from Theorem~\ref{thm:Prime1} and Theorem~\ref{thm:Prime2}
after summation over $p$, respectively. 

\section{Possible Extensions}

Here we point out several further results which can be obtained within our methods. For example, we can estimate the sums
\begin{equation}
\begin{split}
\label{eq:Mob func}
 & \sum_{\substack{1 \le t \le L\\
\Delta(t) \not \equiv 0 \pmod p}} | \mu(t) |
\sym_n\(\psi_p(E(t))\), \\
& \sum_{\substack{1 \le t \le L\\
\Delta(t) \not \equiv 0 \pmod p}} \mu(t)
\sym_n\(\psi_p(E(t))\),
\end{split}
\end{equation}
with the M\"obius function $\mu$. 
Note that the first sum in~\eqref{eq:Mob func} is the sum over
squarefree numbers and can be reduced to the sums of  Lemma~\ref{lem:Interv}
via the standard inclusion-exclusion principle. It correponds to a form  of the Sato--Tate conjecture on average for  curves of the family~\eqref{eq:Family AB}  with  squarefree values of the 
parameter $t$.  
For the second sum in~\eqref{eq:Mob func}
we can use the following analogue of the Vaughan identity
given in Lemma~\ref{lem:Vau}: for any complex-valued function $\psi(t)$ and any real numbers $K,M\ge 1$
with $KM\le L$ and $L\ge 2$, we have
$$
\sum_{t=1}^L\mu(t)\psi(t)\ll
\Omega_1+\Omega_2+\Omega_3+\Omega_4,
$$
where
\begin{align*}
\Omega_1 & =  \left|\sum_{t\le \max \{K,M\}}\mu(t)\psi(t)\right|,\\
\Omega_2 & =  \sum_{k \le KM}\tau (k)\left|\sum_{m\le L/k}\psi(km)\right|,\\
\Omega_3 & =  0,\\
\Omega_4 & = \left|
\sum_{M<m \le L/K} \mu(m)  \sum_{K<k \le L/m}
 \Big| \sum_{d \mid k ,\ d \leq  K }\mu(d) \Big| \psi(km)\right|;
\end{align*}
see the proof of~\cite[Theorem~5.1]{BCFS}.\footnote{We take the opportunity to note that in the proof of \cite[Theorem~5.1]{BCFS}, there are some absolute value symbols that should be brackets; this is 
inconsequential for the argument.} 
So we can
now  proceed as in the proof
of  Lemma~\ref{lem:Lambda}.

\section*{Acknowledgements}

The authors are grateful to Ping Xi for interesting discussions. 
The authors also would like to thank the referee for valuable comments.

The research of the first author was supported by an IUF junior, the second and third authors were supported by the Australian
Research Council Grant DP130100237, and the research of the fourth author
was supported by the Simons Foundation Grant \#234591.

\end{document}